\documentclass{amsart}
\usepackage{amsmath, amssymb, amsthm}
\usepackage{amscd}
\usepackage{graphicx} 
\usepackage{enumerate}

\def\CC{{\mathbb C}}

\def\QQ{{\mathbb Q}}
\def\PP{{\mathbb P}}
\def\QQ{{\mathbb Q}}
\def\RR{{\mathbb R}}

\def\ZZ{{\mathbb Z}}

\def\Qbar{\overline{\mathbb Q}}

\def\0{{\mathbf 0}}
\def\1{{\mathbf 1}}

\def\Gal{\mathrm{Gal}}

\def\GL{\mathrm{GL}}

\def\supp{\mathrm{supp}}

\usepackage[all]{xy}

\newtheorem{thm}{Theorem}

\newtheorem{lemma}[thm]{Lemma}

\newtheorem{cor}[thm]{Corollary}

\newtheorem*{thm*}{Theorem}
\newtheorem*{alg*}{Algorithm}
\newtheorem*{lemma*}{Lemma}

\theoremstyle{remark}

\newtheorem*{rmk*}{Remark}

\newtheorem*{notation*}{Notation}
\newtheorem{example}[thm]{Example}
\newtheorem*{example*}{Example}

\theoremstyle{definition}

\newtheorem*{defn*}{Definition}

\newcommand{\mybf}{\mathbb}

\newcommand{\bP}{\mybf{P}}
\newcommand{\bR}{\mybf{R}}

\newcommand{\bC}{\mybf{C}}

\newcommand{\bQ}{\mybf{Q}}

\newcommand{\al}{\alpha}

\providecommand{\abs}[1]{\lvert#1\rvert}

\def\talltareesidedbox#1{\setbox0=\hbox{$#1$}\dimen0=\wd0 \advance\dimen0 by3pt\rlap{\hbox{\vrule height10pt width.4pt
 depth2pt \kern-.4pt\vrule height10.4pt width\dimen0 depth-10pt\kern-.4pt \vrule height10pt width.4pt depth2pt}}
 \relax \hbox to\dimen0{\hss$#1$\hss}}
\def\tareesidedbox#1{\setbox0=\hbox{$#1$}\dimen0=\wd0 \advance\dimen0 by3pt\rlap{\hbox{\vrule height8pt width.4pt
 depth2pt \kern-.4pt\vrule height8.4pt width\dimen0 depth-8pt\kern-.4pt \vrule height8pt width.4pt depth2pt}}
\relax \hbox to\dimen0{\hss$#1$\hss}}
\def\shorttareesidedbox#1{\setbox0=\hbox{$#1$}\dimen0=\wd0 \advance\dimen0 by3pt\rlap{\hbox{\vrule height7pt width.4pt
 depth2pt \kern-.4pt\vrule height7.4pt width\dimen0 depth-7pt\kern-.4pt \vrule height7pt width.4pt depth2pt}}
 \relax \hbox to\dimen0{\hss$#1$\hss}}

\newcommand{\har}{h_{\mathrm{Ar}}}
\newcommand{\hweil}{h_{\mathrm{Weil}}}

\title{Energy integrals and small points for the Arakelov height}

\author[Fili]{Paul Fili}
\address{Department of Mathematics\\ Oklahoma State University, Stillwater, OK 74078}
\email{fili@post.harvard.edu}
\author[Petsche]{Clayton Petsche}
\address{Department of Mathematics\\ Oregon State University, Corvallis, OR 97331}
\email{petschec@math.oregonstate.edu}
\author[Pritsker]{Igor Pritsker}
\address{Department of Mathematics\\ Oklahoma State University, Stillwater, OK 74078}
\email{igor@math.okstate.edu}

\date{\today}

\usepackage{hyperref}

\begin{document}

\begin{abstract}
We study small points for the Arakelov height on the projective line.  First, we identify the smallest positive value taken by the Arakelov height, and we characterize all cases of equality.  Next we solve several archimedean energy minimization problems with respect to the chordal metric on the projective line, and as an application, we obtain lower bounds on the Arakelov height in fields of totally real and totally $p$-adic numbers.
\end{abstract}

\maketitle

\section{Introduction}

The Arakelov height function $\har:\PP^1(\Qbar)\to\RR$ is defined by the formula 
\begin{equation}\label{ArakelovDef}
\har(\alpha)=\frac{1}{[K:\QQ]}\sum_{v\in M_\QQ}\sum_{\sigma:K\hookrightarrow\CC_v}\log\|(\sigma(\alpha_0),\sigma(\alpha_1))\|_v.
\end{equation}
Here $\alpha=(\alpha_0:\alpha_1)$ is a point in $\PP^1(\Qbar)$, $K$ is any number field containing the coordinates of $\alpha$, and the outer sum ranges over the set $M_\QQ=\{\infty,2,3,5,7,\dots\}$ of all places $v$ of $\QQ$.  For each place $v\in M_\QQ$, the inner sum in $(\ref{ArakelovDef})$ ranges over all $[K:\QQ]$ distinct embeddings $\sigma:K\hookrightarrow\CC_v$, and the norm $\|\cdot\|_v:\CC_v^2\to\RR$ is defined by
\begin{equation}\label{LocalNormDef}
\|(x_0,x_1)\|_v = 
\begin{cases}
(|x_0|_v^2+|x_1|_v^2)^{1/2} & \text{ if $v=\infty$} \\
\max(|x_0|_v,|x_1|_v) & \text{ if $v=2,3,5,7,\dots$}
\end{cases}
\end{equation}
where $|\cdot|_v$ denotes the absolute value on $\CC_v$, normalized to coincide with either the standard real or $p$-adic absolute value when restricted to $\QQ$.  The product formula
\begin{equation}\label{ProdForm}
\sum_{v\in M_\QQ}\sum_{\sigma:K\hookrightarrow\CC_v}\log|\sigma(x)|_v=0 \hskip2cm (x\in K^\times)
\end{equation}
ensures that the value of $\har(\alpha)$ does not depend on the choice of homogeneous coordinates for $\alpha$, and standard properties of field extensions ensure that the value of $\har(\alpha)$ does not depend on the choice of the number field $K$.  For more information on $\har$ see, e.g., \cite{bombierigubler} $\S$~2.8 or \cite{MR1491857}.

The definition of $\har$ should be compared with the definition of the standard Weil height function $\hweil$.  To obtain $\hweil$, rather than the formula $(\ref{LocalNormDef})$ involving the $\ell^2$-norm at the archimedean place, one instead uses the sup-norm $\|(x_0,x_1)\|_v =\max(|x_0|_v,|x_1|_v)$ at both the archimedean and non-archimedean places.  Many authors have studied questions surrounding points of small Weil height and their properties; for example, these questions are closely related to the well-known question of Lehmer \cite{MR1503118} on polynomials of small Mahler measure, results of Schinzel \cite{MR0360515} and Bombieri-Zannier \cite{MR1898444} on small totally real and totally $p$-adic points, and the theorem of Bilu \cite{MR1470340} on the equidistribution of points of small Weil height.

Despite the fact that the Arakelov height is perhaps the second most well-studied elementary example of an absolute height function on projective space (in the sense of Weil), the topic of small points with respect to $\har$ has been comparatively neglected.  

We first make a trivial observation.  Identifying $\PP^1(\Qbar)=\Qbar\cup\{\infty\}$ by setting $\alpha=(\alpha:1)$ and $\infty=(1:0)$, it follows at once from the definition that $\har(\alpha)\geq0$ for all $\alpha\in\PP^1(\Qbar)$, with $\har(\alpha)=0$ if and only if $\alpha=0$ or $\infty$.  (Compare with $\hweil$, which vanishes precisely at $0,\infty$, and the roots of unity.)  

Our first result specifies the next smallest value taken on by the Arakelov height and characterizes all of the (infinitely many) cases of equality.  

\begin{thm}\label{ElementaryLBThm}
The lower bound $\har(\alpha)\geq\frac{1}{2}\log 2$ holds for all $\alpha\in\PP^1(\Qbar)\setminus\{0,\infty\}$.  Equality $\har(\alpha)=\frac{1}{2}\log 2$ holds if and only if $\alpha$ is a root of unity.
\end{thm}

Next, we turn to the question of giving lower bounds for $\har(\alpha)$ under the assumption that the conjugates of $\al\in\Qbar$ satisfy certain splitting conditions.  For example, if all complex embeddings of $\alpha$ lie in $\bR$, one says that $\al$ is totally real; similarly, if all embeddings of $\al$ into $\CC_p$ lie in $\bQ_p$, one says that $\al$ is totally $p$-adic.  In \cite{FiliPetsche}, the first two authors used potential theoretic techniques to minimize a certain energy integral, and used this to obtain lower bounds on the Weil height $\hweil(\alpha)$ under splitting conditions.  

In the present paper, we carry out similar investigations for the Arakelov height.  Recall that, for each place $v\in M_\QQ$, the standard projective metric is defined by 
\begin{equation*}
\delta_v:\PP^1(\CC_v)\times\PP^1(\CC_v)\to[0,1] \hskip1cm \delta_v(x,y) =\frac{|x_0y_1-y_0x_1|_v}{\|(x_0,x_1)\|_v\|(y_0,y_1)\|_v}
\end{equation*}
for $x=(x_0:x_1)$ and $ y=(y_0:y_1)$ in $\PP^1(\CC_v)$, where the norm $\|\cdot\|_v:\CC_v^2\to\RR$ is defined in $(\ref{LocalNormDef})$.  In the archimedean case, the metric $\delta_v$ coincides with (one half of) the chordal distance obtained by identiying $\PP^1(\CC)$ with the unit sphere in $\RR^3$ via stereographic projection.  

If $K$ is a number field and $\alpha,\beta\in\PP^1(K)$ are distinct, then $(\ref{ArakelovDef})$ and $(\ref{ProdForm})$ give
\begin{equation}\label{ArakelovDefSum}
\har(\alpha)+\har(\beta)=\frac{1}{[K:\QQ]}\sum_{v\in M_\QQ}\sum_{\sigma:K\hookrightarrow\CC_v}-\log\delta_v(\sigma(\alpha),\sigma(\beta)).
\end{equation}
In particular, given a point $\alpha\in\PP^1(\Qbar)$ of degree $d=[\QQ(\alpha),\QQ]\geq2$, let $\{\alpha_1,\dots,\alpha_d\}$ be the complete set of $\Gal(\Qbar/\QQ)$-conjugates of $\alpha$ in $\PP^1(\Qbar)$.  For each place $v\in M_\QQ$ we may view $\{\alpha_1,\dots,\alpha_d\}$ as a subset of $\PP^1(\CC_v)$ via some fixed embedding $\iota:\Qbar\hookrightarrow\CC_v$.  Then $(\ref{ArakelovDefSum})$ and the $\Gal(\Qbar/\QQ)$-invariance of the Arakelov height implies that 
\begin{equation}\label{ArakelovDefEnergySum}
\har(\alpha)=\frac{1}{2}\sum_{v\in M_\QQ}D_v(\alpha),
\end{equation}
where for each place $v\in M_\QQ$ we define the {\em energy sum} \begin{equation}\label{LocalEnergySum}
D_v(\alpha)=\frac{1}{d(d-1)}\sum_{\stackrel{1\leq i,j\leq d}{i\neq j}}-\log\delta_v(\alpha_i,\alpha_j).
\end{equation}

For large $d$, the energy sum $(\ref{LocalEnergySum})$ may be viewed as a discrete approximation to the energy integral
\begin{equation}\label{eqn:EnergyIntegral}
 I_v(\nu) = \iint_{\bP^1(\bC_v)\times \bP^1(\bC_v)} -\log \delta_v(x,y)\,d\nu(x)\,d\nu(y)
\end{equation}
associated to a unit Borel measure\footnote{For technical reasons, in the non-archimedean case it is better to consider the energy integral for Borel measures supported on the Berkovich projective line.  Since we focus here on the archimedean case, this subtlety does not affect the results of this paper.} $\nu$ on $\PP^1(\CC_v)$.  It is the passage from the sum $(\ref{LocalEnergySum})$ to the integral $(\ref{eqn:EnergyIntegral})$ which opens the door to potential theoretic techniques used in \cite{FiliPetsche} and in the present paper.  Local height sums similar to $(\ref{LocalEnergySum})$ have been considered by many authors in different contexts, notably in Arakelov intersection theory and arithmetic geometry, but the connections with potential theory have their origins in Rumely \cite{MR1710794} and Baker-Rumely \cite{MR2244226}.

In the non-archimedean case $v=p<\infty$, one has the trivial lower bounds $D_p\geq0$  and $I_p\geq0$ owing to the nonegativity of the potential kernel $-\log \delta_p(x,y)$.  But if one considers only those unit Borel measures $\nu$ supported on $\PP^1(\QQ_p)$, it was shown in Fili-Petsche \cite{FiliPetsche} that \begin{equation}\label{FiliPetscheNonArch}
I_p(\nu)\geq\frac{p\log p}{p^2-1}
\end{equation} 
with equality if and only if $\nu$ is the unique $\GL_2(\ZZ_p)$-invariant unit Borel measure on $\PP^1(\QQ_p)$.

Thus we turn our attention to the archimedean case $\CC_v=\CC$, and we drop the subscript $v$ to ease notation.  It is shown in \cite{MR0414898} $\S$~III.11 that if $E$ is a compact subset of $\PP^1(\CC)$ which is large enough (in a precise potential-theoretic sense), then there exists a unique Borel probability measure $\mu_E$ supported on $E$ which minimizes the energy integral $I(\nu)$ among all Borel probability measures $\nu$ supported on $E$.  We will give further potential-theoretic details in $\S$~\ref{PotentialTheorySect}.  The following results give explicit calculations for some sets $E$ of arithmetic interest.

\begin{thm}\label{thm:C}
When $E=\PP^1(\CC)$, the unique minimal energy measure $\mu_{\CC}$ on $\PP^1(\bC)$ is given explicitly by
\begin{equation}\label{ComplexMEMeas}
d\mu_{\CC}(z)=\frac{1}{\pi(1+|z|^2)^2}d\ell(z)
\end{equation}
where $\ell(z)$ is Lebesgue measure on $\CC=\PP^1(\CC)\setminus\{\infty\}$. The associated minimal energy is
\begin{equation*}
I(\mu_{\CC})=\frac{1}{2}.
\end{equation*}
\end{thm}
We note that Theorem \ref{thm:C} appeared in a slightly different form in \cite[Proposition 5.2]{MR627747}.
\begin{thm}\label{thm:R}
When $E=\PP^1(\RR)$, the unique minimal energy measure $\mu_{\RR}$ on $\PP^1(\RR)$ is given explicitly by
\begin{equation}\label{RealMEMeas}
d\mu_{\RR}(z)=\frac{1}{\pi(1+x^2)}d\ell(x),
\end{equation}
where $\ell(x)$ is Lebesgue measure on $\RR=\PP^1(\RR)\setminus\{\infty\}$. The associated minimal energy is
\begin{equation*}
I(\mu_{\RR})=\log 2=0.69315....
\end{equation*}
\end{thm}

In order to state the result for more general subsets of the real line, we need the notion of harmonic measure, see \cite[Section 4.3]{MR1334766}. For a compact set $E\subset \RR$ of positive logarithmic capacity, let $\Omega=\PP^1(\CC)\setminus E$ be the complementary domain containing $\infty$ and let $g_{\Omega}(z,\infty)$ be the Green function of $\Omega$ with pole at $\infty.$ Given a point $z\in\Omega$, we denote the harmonic measure $\omega_{\Omega}(z,B)$ of a Borel set $B\subset\partial\Omega$ at $z$ with respect to $\Omega.$

\begin{thm}\label{thm:R-interval}
For any compact set $E\subset \RR$ of positive logarithmic capacity, the unique minimal energy measure is given by
\[
\mu_E = \omega_{\Omega}(i,\cdot) \quad \mbox{with} \quad I(\mu_E) = g_{\Omega}(i,\infty) + \frac{1}{2} \int \log(1+x^2)\,d\omega_{\Omega}(i,x).
\]
In particular, if $E=[-r,r]$ is an interval, then
\begin{align*}
d\mu_{[-r,r]}(x) &= \frac{(\sqrt{r^2+1}+1)\,dx}{\pi\sqrt{r^2-x^2}(x^2+(\sqrt{r^2+1}+1-\sqrt{r^2-x^2})^2)} \\ &+ \frac{(\sqrt{r^2+1}+1)\,dx}{\pi\sqrt{r^2-x^2}(x^2+(\sqrt{r^2+1}+1+\sqrt{r^2-x^2})^2)},\ x\in(-r,r).
\end{align*}
and
\[
I(\mu_{[-r,r]}) = \log\frac{2\sqrt{r^2+1}}{r}.
\]
\end{thm}

Finally, we return to the Arakelov height and describe an application of the above potential-theoretic results to giving lower bounds $\har(\alpha)$ for points $\alpha$ satisfying splitting conditions.

\begin{thm}\label{TotallyvadicThm}
Let $S$ be a subset of the set $M_\QQ$ of all places of $\QQ$.  Let $L_S$ be the subfield of $\Qbar$ consisting of all algebraic numbers which are totally $v$-adic for all places $v\in S$.  If $\infty\notin S$ then 
\begin{equation}\label{TotallyvadicBound1}
\liminf_{\alpha\in\PP^1(L_S)}\har(\alpha) \geq\frac{1}{4}+\frac{1}{2}\sum_{p\in S}\frac{p\log p}{p^2-1}.
\end{equation}
If $\infty\in S$ then 
\begin{equation}\label{TotallyvadicBound2}
\liminf_{\alpha\in\PP^1(L_S)}\har(\alpha) \geq\frac{1}{2}\log 2+\frac{1}{2}\sum_{p\in S\setminus\{\infty\}}\frac{p\log p}{p^2-1}.
\end{equation}
\end{thm}

Naturally, the lower bound of Theorem~\ref{TotallyvadicThm} is only interesting when it is greater than the elementary and unconditional lower bound $\frac{1}{2}\log2$ of Theorem~\ref{ElementaryLBThm}, and this depends on the set $S$.  For example, when $S$ is empty the lower bound of Theorem~\ref{TotallyvadicThm} is $\frac{1}{4}$, which is worse than $\frac{1}{2}\log2$.  When $S=\{v\}$, a single place, the lower bound of Theorem~\ref{TotallyvadicThm} beats $\frac{1}{2}\log2$ when $v=2, 3, 5, 7, 11, 13$ but not when $v=\infty,17, 19,\dots$.  When $S=\{\infty,p\}$ for a prime $p$, the lower bound of Theorem~\ref{TotallyvadicThm} always beats $\frac{1}{2}\log2$.  When $S=\{p, q\}$ for distinct primes $p<q$, the lower bound of Theorem~\ref{TotallyvadicThm} always beats $\frac{1}{2}\log2$ when $p=2, 3, 5, 7, 11, 13$, and there are exactly 82 pairs of distinct primes $S=\{p,q\}$, with $13<p<q$, for which the lower bound of Theorem~\ref{TotallyvadicThm} beats $\frac{1}{2}\log2$.

The following result should be viewed as a refinement of $(\ref{TotallyvadicBound2})$ in the totally real case, in which all of the $\Gal(\Qbar/\QQ)$-conjugates are contained in a symmetric real interval.

\begin{thm}\label{TotallyvadicIntervalThm}
Let $S$ be a subset of the set $M_\QQ$ of all places of $\QQ$, and assume that $\infty\in S$.  For $r>0$, let $L_{S,r}$ be the subset of $\Qbar$ consisting of all algebraic numbers which are totally $p$-adic for all primes $p\in S$, and totally real with all $\Gal(\Qbar/\QQ)$-conjugates in $\PP^1(\CC)$ lying in the interval $[-r,r]$.  Then 
\begin{equation}\label{TotallyvadicBound3}
\liminf_{\alpha\in\PP^1(L_{S,r})}\har(\alpha) \geq \frac{1}{2}\log\frac{2\sqrt{r^2+1}}{r}+\frac{1}{2}\sum_{p\in S}\frac{p\log p}{p^2-1}.
\end{equation}
\end{thm}

We give some examples to illustrate the applications of our results.
\begin{example}
 Suppose that $S=\{2,\infty\}$, so that $L_S$ is the field of all algebraic numbers which are totally $2$-adic and totally real. It then follows from Theorem \ref{TotallyvadicThm} that
 \begin{align*}
  \liminf_{\alpha\in\PP^1(L_{S})}\har(\alpha) &\geq \frac12 \log 2 + \frac{1}{2}\cdot\frac{2\log 2}{2^2-1}\\ &= 0.346574\ldots + 0.231049\ldots = 0.577623\ldots
 \end{align*}
\end{example}
\begin{example}
 Now suppose we impose the additional restriction to the previous example that all conjugates in $\bC$ also lie in the interval $[-2,2]$, that is, we take $S=\{2,\infty\}$ and $r=2$, so that $L_{S,r}$ consists of all algebraic numbers which are totally $2$-adic and have all conjugates lying in $[-2,2]$ in $\bC$. It then follows from Theorem \ref{TotallyvadicIntervalThm} that
 \begin{align*}
  \liminf_{\alpha\in\PP^1(L_{S,r})}\har(\alpha) &\geq \frac12 \log \frac{2\sqrt{2^2 + 1}}{2} + \frac{1}{2}\cdot\frac{2\log 2}{2^2-1}\\
  &= 0.402359\ldots + 0.231049\ldots = 0.633409\ldots
 \end{align*}
\end{example}
\begin{example}
 Suppose that $S=\{\infty\}$ and $r=2$, so that we consider all algebraic numbers which are totally real with conjugates lying in the interval $[-2,2]$. Then Theorem \ref{TotallyvadicIntervalThm} implies that 
 \[
  \liminf_{\alpha\in\PP^1(L_{S,r})}\har(\alpha) \geq \frac12 \log \frac{2\sqrt{2^2 + 1}}{2} = 0.402359\ldots
 \]
 We note that if we imposed the additional restriction that our numbers were algebraic integers (which is not an assumption of our results above) then we have by well-known equidistribution results that the conjugates of these numbers will equidistribute in the interval $[-2,2]$ according to its logarithmic equilibrium measure $d\nu(x) = dx/\pi \sqrt{4-x^2}$, and therefore the Arakelov height will in fact limit to
 \[
  \int_{-2}^2 \frac{\log \sqrt{1+x^2}}{\pi \sqrt{4-x^2}}dx = 0.481212\ldots
 \]
 and this limit is achieved for these numbers (which are preperiodic points for the Chebyshev map $T_2(x)=x^2-2$). It is an interesting open question to determine when the lower bounds for the Arakelov height in Theorems \ref{TotallyvadicThm} and \ref{TotallyvadicIntervalThm} are achieved as limits of heights for sequences of algebraic numbers.
\end{example}

\section{An elementary lower bound on the Arakelov height}\label{ElementarySect}

In this section we prove Theorem~\ref{ElementaryLBThm}.  The proof is an application of this elementary lemma.

\begin{lemma}\label{ElementaryLemma}
Let $d\geq1$ and $r\geq0$ be integers, and let $a_1,\dots,a_d$ and $b_1,\dots,b_r$ be positive real numbers with $a_1\dots a_db_1 \dots b_r=1$.  Then the quantity
$$
L=L(a_1,\dots,a_d,b_1,\dots,b_r)=\frac{1}{d}\bigg(\sum_{j=1}^d\frac{1}{2}\log(1+a_j^2)+\sum_{k=1}^r\log^+b_k\bigg)
$$
satisfies the lower bound $L\geq\frac{1}{2}\log2$. Equality $L=\frac{1}{2}\log2$ holds if and only if $a_1=\dots=a_d=b_1=\dots=b_r=1$.
\end{lemma}
\begin{proof}
We are going to use the inequality 
\begin{equation}\label{AuxIneq}
\frac{a^{2t}}{1+a^{2t}}\log a\geq\frac{1}{2}\log a
\end{equation}
which holds for all $0\leq t\leq 1$ and $a>0$.  To check $(\ref{AuxIneq})$ one uses the bounds $\frac{x}{1+x}\leq \frac{1}{2}$ (for $0\leq x\leq 1$) and $\frac{x}{1+x}\geq \frac{1}{2}$ (for $x\geq 1$).

Now define $f(t)=L(a_1^t,\dots,a_d^t,b_1^t,\dots,b_r^t)$ for $0\leq t\leq 1$.  Then using $(\ref{AuxIneq})$ and the trivial inequalities $\log^+b\geq\frac{1}{2}\log^+b\geq\frac{1}{2}\log b$ for $b>0$, we have
\begin{equation*}
\begin{split}
f'(t) & = \frac{1}{d}\bigg(\sum_{j=1}^d\frac{a_j^{2t}}{1+a_j^{2t}}\log a_j +\sum_{k=1}^r\log^+b_k\bigg) \\
	& \geq \frac{1}{2d}\bigg(\sum_{j=1}^d\log a_j +\sum_{k=1}^r\log b_k\bigg)=0.
\end{split}
\end{equation*}
It follows that $f(t)$ is nondecreasing for $0\leq t\leq 1$ and therefore 
\begin{equation}\label{DesiredLB}
L=f(1)\geq f(0)=\frac{1}{2}\log2.
\end{equation}

Now suppose that $L=\frac{1}{2}\log 2$.  Then $(\ref{DesiredLB})$ and the fact that $f(t)$ is nondecreasing implies that $f(t)$ is constant.  Therefore
\begin{equation*}
\begin{split}
f''(t) & = \frac{1}{d}\sum_{j=1}^d\frac{a_j^{2t}}{(1+a_j^{2t})^2}2(\log a_j)^2  =0
\end{split}
\end{equation*}
for all $0\leq t\leq 1$, which can occur only if $\log a_j=0$ for all $1\leq j\leq d$.  Finally, from the definition of $L$, the assumption that $L=\frac{1}{2}\log 2$, and the fact that $a_1=\dots=a_d=1$, we deduce $\sum_{k=1}^r\log^+b_k=0$.  It follows that $b_k\leq 1$ for all $1\leq k\leq r$.  But since $b_1\dots b_r=1$ we conclude $b_1=\dots =b_r=1$.
\end{proof}

\begin{proof}[Proof of Theorem~\ref{ElementaryLBThm}]
For $\alpha\in\Qbar^\times$, set $K=\QQ(\alpha)$ and $d=[K:\QQ]$.  Viewing $\alpha$ as the point $(\alpha:1)$ in  $\PP^1(K)$, the definition $(\ref{ArakelovDef})$ simplifies to 
\begin{equation*}
\har(\alpha) = \frac{1}{d}\bigg( \sum_{\sigma:K\hookrightarrow\CC}\frac{1}{2}\log (1+|\sigma(\alpha)|^2) + \sum_{p}\sum_{\sigma:K\hookrightarrow\CC_p}\log^+|\sigma(\alpha)|_p \bigg).
\end{equation*}
The lower bound $\har(\alpha)\geq\frac{1}{2}\log 2$ now follows immediately from Lemma~\ref{ElementaryLemma}, taking the $a_j$ to be the numbers $|\sigma(\alpha)|$ as $\sigma$ ranges over all complex embeddings of $K$, and taking the $b_k$ to be the numbers $|\sigma(\alpha)|_p$ as $p$ ranges over a sufficiently large finite set of rational primes, and for each $p$, $\sigma$ ranges over all embeddings of $K$ into $\CC_p$.  The product formula $(\ref{ProdForm})$ ensures that the hypothesis $a_1\dots a_db_1 \dots b_r=1$ of Lemma \ref{ElementaryLemma} is satisfied.  If equality $\har(\alpha)=\frac{1}{2}\log 2$ holds, then Lemma \ref{ElementaryLemma} says that $|\sigma(\alpha)|_v=1$ for all places $v\in M_\QQ$ and all embeddings $\sigma:K\hookrightarrow\CC_v$; by Kronecker's theorem this occurs only when $\alpha$ is a root of unity.
\end{proof}

\section{Minimal energy calculations}\label{PotentialTheorySect}
First, let us set up the energy problem. For a Borel probability measure $\mu$ on $\bP^1(\bC)$, let
\[
 I(\mu) = \iint_{\bP^1(\bC)^2} -\log \delta(x,y)\,d\mu(x)\,d\mu(y).
\]
We call this integral the elliptic energy of $\mu$. Tsuji established some of the fundamental potential-theoretic results for the above kernel $\delta(x,y)$ (cf. \cite{MR0027861}). It follows from \S III.11 in \cite{MR0414898} and \cite[Theorem III.8]{MR0414898} that, for any closed subset $E\subset \bP^1(\bC)$, if the elliptic Robin constant $V_\delta(E)$ given by
\[
 V_\delta(E) = \inf_{\substack{\mu\\ \supp(\mu)\subseteq E}} I_\delta(\mu)
\]
is finite, then there exists a unique measure $\mu=\mu_E$ such that $I(\mu) = V_\delta(E)$.
\begin{thm}
 Given $\nu$ a Borel probability measure supported on the closed set $E\subseteq \bP^1(\bC)$,
 \[
  \inf_{x\in E} U^\nu_\delta(x) \leq V_\delta(E) \leq \sup_{x\in E} U^\nu_\delta(x)
 \]
 where $U^\nu_\delta(x)$ is the elliptic potential function
 \[
  U^\nu_\delta(x) = \int_E -\log \delta(x,y)\,d\nu(y).
 \]
\end{thm}
\begin{proof}
Our proof follows the same lines as the argument in \cite[Theorem III.15]{MR0414898} and \cite[Theorem 7]{FiliPetsche}. Since $-\log \delta(x,y)\geq 0$ on $\bP^1(\bC)$ and $\mu,\nu$ are probability measures, it follows from Tonelli's theorem that
 \[
  \int_E U^\nu_\delta(x) \,d\mu(x) = \int_E U^\mu_\delta(x)\,d\nu(x).
 \]
 Further, it follows from \cite[Theorem III.46]{MR0414898} that $U^\mu_\delta(x) \leq V_\delta(E)$ everywhere, so
 \[
  \int_E U^\nu_\delta(x) \,d\mu(x)\leq V_\delta(E)
 \]
 hence $\inf_{x\in E} U^\nu_\delta(x) \leq V_\delta(E)$. In other direction, we may as well assume that that $U^\nu_\delta(x)<\infty$, and from the usual maximum principle argument it follows that $I(\nu)<\infty$. Following the proof of \cite[Theorem III.7]{MR0414898}, the only change being replacing the logarithmic kernel with the kernel $-\log \delta(x,y)$, that $\nu$ cannot assign any positive measure to a polar set (that is, a set of capacity zero). Therefore, since $U^\mu_\delta(x) = V_\delta(E)$ quasi everywhere for $x\in E$, it follows that
 \[
  \int_E U^\mu_\delta(x)\,d\nu(x) = V_\delta(E),
 \]
 and hence we can conclude that $\sup_{x\in E} U^\nu_\delta(x) \geq V_\delta(E)$.
\end{proof}
From the uniqueness of the elliptic equilibrium measure and the above lemma, we immediately gain the following corollary:
\begin{cor}\label{cor:constant-potential-is-equil}
 If $\nu$ is a Borel probability measure supported on a closed set $E\subset\bP^1(\bC)$ and the elliptic potential $U^\nu_\delta(x)=C$ is constant for all $x\in E\setminus F$, where $F$ is a polar subset, then $\nu$ is the elliptic equilibrium measure of $E$ and $V_\delta(E)=C$.
\end{cor}

\begin{proof}[Proof of Theorem \ref{thm:C}]
Let $\mu$ be the measure on $\PP^1(\CC)$ described on the right-hand-side of $(\ref{ComplexMEMeas})$.  By Corollary~\ref{cor:constant-potential-is-equil}, in order to show that $\mu$ is the minimal energy measure, it suffices to check that the potential function $U_\delta^\mu(x)$ is constant on $\PP^1(\CC)$.  The group $U(2,\CC)$ of $2\times2$ unitary matrices acts transitively as a group of isometries on $\PP^1(\CC)$ with resepct to the projective metric $\delta(x,y)$.  Further, a straightforward calculation shows that the measure $\mu$ is $U(2,\CC)$-invariant.  Therefore, given $x,x'\in\PP^1(\CC)$, select $f\in U(2,\CC)$ for which $f(x')=x$, and we have 
\begin{equation*}
\begin{split}
U_\delta^\mu(x') & =\int_{\PP^1(\CC)}-\log\delta(f^{-1}(x),y)\,d\mu(y) \\
	& =\int_{\PP^1(\CC)}-\log\delta(x,f(y))\,d\mu(y) \\
	& =\int_{\PP^1(\CC)}-\log\delta(x,y)\,d(f_*\mu)(y) \\
	& =\int_{\PP^1(\CC)}-\log\delta(x,y)\,d\mu(y) =U_\delta^\mu(x).
\end{split}
\end{equation*}
It follows that $U_\delta^\mu(x)$ is constant, as desired.  To calculate the minimal energy, again using Corollary~\ref{cor:constant-potential-is-equil} we have 
\begin{equation*}
\begin{split}
I(\mu) & = U_\delta^\mu(\infty) \\
  & = \int_{\PP^1(\CC)}-\log\delta(z,\infty)\,d\mu(z) \\
  & = \int_{\PP^1(\CC)}\frac{1}{2}\log(1+|z|^2)\,d\mu(z) \\
  & = \int_{0}^{2\pi}\int_{0}^{\infty}\frac{1}{2}\log(1+r^2)\frac{r\,dr\,d\theta}{\pi(1+r^2)^2} =\frac{1}{2}.
\end{split}
\end{equation*}

One can also prove this result by using Theorem 6.1 of \cite[p. 245]{MR1485778}.
\end{proof}

\begin{proof}[Proof of Theorem \ref{thm:R}]
The orthogonal group $O(2,\RR)$ acts transitively as a group of isometries on $\PP^1(\RR)$ with resepct to the projective metric $\delta(x,y)$, and a calculation shows that the measure $\mu$ on $\PP^1(\RR)$ described on the right-hand-side of $(\ref{RealMEMeas})$ is $O(2,\RR)$-invariant.  It follows from the same argument as in the complex case that $\mu$ is the minimal energy measure supported on $\PP^1(\RR)$.  The minimal energy is
\begin{equation*}
\begin{split}
I(\mu) & = U_\delta^\mu(\infty) \\
	 & = \int_{\PP^1(\RR)}-\log\delta(x,\infty)\,d\mu(x) \\
	 & = \frac{1}{\pi}\int_{0}^{\infty}\log(1+x^2)\frac{dx}{(1+x^2)} \\
	 & = \frac{2}{\pi}\int_{0}^{\pi/2}\log(\sec\theta) d\theta =\log 2.
\end{split}
\end{equation*}

We give a second proof that suggests how one can deal with general sets on the real line. Using the Poisson formula for the function $u_w(z)=\log|z-w|,\ \Im(w)<0,$ which is harmonic in the upper half plane, we have
\begin{equation} \label{PoissonR}
u_w(z) = \int_\RR \frac{y\,u_w(x-t)\,dt}{\pi(t^2+y^2)},\quad z=x+iy,\ y>0.
\end{equation}
In particular, if we let $y=1$ and $w=-si$ for $s>0$, then
\begin{equation} \label{PoissonR-2}
 u_{-si}(x+i) = \log\,\abs{x+i + si} = \int_\RR \frac{\log\,\abs{x - t + si}\,dt}{\pi(1+t^2)}
\end{equation}
As this is a logarithmic potential with respect to the Borel measure $\mu_\RR$, it follows from the continuity of logarithmic potentials (see Corollary 5.6 in \cite[p. 61]{MR1485778}) that this equality is also valid for $s=0$, and thus we find that
\begin{equation}\label{equilR}
 \frac{1}{2}\log(1+x^2) =\log\,\abs{x+i} = \int_\RR \frac{\log\,\abs{x - t}\,dt}{\pi(1+t^2)},\quad x\in\RR.
\end{equation}
Thus Corollary \ref{cor:constant-potential-is-equil} confirms that $\mu_\RR$ is indeed the elliptic equilibrium measure of $\RR,$ as 
\begin{align*}
U^{\mu_\RR}_\delta(s) &= - \int_\RR \frac{\log|x-s|\,dx}{\pi(1+x^2)} + \frac{1}{2}\log(1+s^2) + \frac{1}{2} \int \frac{\log(1+x^2)\,dx}{\pi(1+x^2)} \\ &= \int_\RR \frac{\log(1+x^2)\,dx}{2\pi(1+x^2)},\quad s\in\RR. 
\end{align*}
It follows as before that
\begin{align*}
I(\mu_\RR) = \int_\RR \frac{\log(1+x^2)\,dx}{2\pi(1+x^2)} = \int_\RR \frac{u_{-i}(x)\,dx}{\pi(1+x^2)} = u_{-i}(i) = \log 2,
\end{align*}
where the integral above is computed from the Poisson formula \eqref{PoissonR} for $u_{-i}(z)=\log|z+i|.$

A third proof of this result could be given using the Hilbert transform, following an argument parallel to the proof of Theorem 1(b) of \cite{FiliPetsche}; we do not work out the details here.
\end{proof}

Equation \eqref{equilR} indicates that the potential of the unit point mass $\delta_i$ coincides with the  potential of $\mu_\RR$ on the real line, i.e., $\mu_\RR$ is the balayage of $\delta_i$ out of the upper half plane in the sense of Theorem 4.1 of \cite[p. 110]{MR1485778}. This idea is extended in the proof of Theorem \ref{thm:R-interval}.

\begin{proof}[Proof of Theorem \ref{thm:R-interval}]
We develop the idea of balayage for the unit point mass $\delta_i$ at $i$ from the previous proof. The balayage of $\delta_i$ from the domain $\Omega=\overline{\CC}\setminus E$ onto $E\subset\RR$ is given by the harmonic measure $\omega_{\Omega}(i,\cdot)$, see Section II.4 of \cite{MR1485778} and Section 4.3 of \cite{MR1334766}. It follows from Theorem 4.4 of \cite[p. 115]{MR1485778} that the logarithmic potential of $\omega_{\Omega}(i,\cdot)$ satisfies
\begin{align*}
&\int \log\frac{1}{|x-t|}\,d\omega_{\Omega}(i,t) + \frac{1}{2}\log(1+x^2) \\ &=  \int \log\frac{1}{|x-t|}\,d\omega_{\Omega}(i,t) -  \int \log\frac{1}{|x-t|}\,d\delta_i(t) = \int g_\Omega(t,\infty)\,d\delta_i(t) = g_\Omega(i,\infty).
\end{align*}
for quasi every $x\in E.$ Hence $\omega_{\Omega}(i,\cdot)$ is the elliptic equilibrium measure of $E$ by Corollary \ref{cor:constant-potential-is-equil}, because the elliptic potential of $\omega_{\Omega}(i,\cdot)$ is constant quasi everywhere on $E$. Furthermore, we obtain that 
\begin{align*}
I(\mu_E) &= \int \left(\int \log\frac{1}{|x-t|}\,d\omega_{\Omega}(i,t) + \frac{1}{2}\log(1+x^2) \right)d\omega_{\Omega}(i,x) \\ &+ \frac{1}{2} \int \log(1+t^2)\,d\omega_{\Omega}(i,t) \\ &= g_\Omega(i,\infty) + \frac{1}{2} \int \log(1+t^2)\,d\omega_{\Omega}(i,t).
\end{align*}

We need to find $\omega_{\Omega}(i,\cdot)$ explicitly for $E=[-r,r],\ r>0.$ This is conveniently done by using conformal invariance of harmonic measures, see Theorem 4.3.8 on \cite[p. 101]{MR1334766}. If $\Phi$ is a conformal mapping of $\Omega=\overline{\CC}\setminus [-r,r]$ onto $\Delta=\overline{\CC}\setminus\{t\in\CC:|t|\le 1\}$ such that $\Phi(i)=\infty,$ then the image of $\omega_{\Omega}(i,\cdot)$ under $\Phi$ is $\omega_{\Delta}(\infty,\cdot) = |dt|/(2\pi)$ supported on $\partial \Delta = \{t\in\CC:|t|= 1\}$. The mapping $\Phi$ may constructed as the composition of two standard conformal mappings. These are   $w=\Phi_1(z)=(z+\sqrt{z^2-r^2})/r$ that maps $\Omega$ onto $\Delta$ with $w_0=\Phi_1(i)=(\sqrt{r^2+1}+1)i/r,$ and $t=\Phi_2(w)=(\overline{w}_0 w -1)/(w-w_0)$ that is a self map of $\Delta$ sending $w_0$ to infinity. Denoting the upper limiting values of $\Phi$ on $[-r,r]$ by $\Phi_+,$ and the lower limiting values by $\Phi_-,$ we obtain the following expression for $d\omega_{\Omega}(i,x)$:
\begin{align*}
\frac{(|\Phi'_+(x)|+|\Phi'_-(x)|)\,dx}{2\pi} &= \frac{(\sqrt{r^2+1}+1)\,dx}{\pi\sqrt{r^2-x^2}(x^2+(\sqrt{r^2+1}+1-\sqrt{r^2-x^2})^2)} \\ &+ \frac{(\sqrt{r^2+1}+1)\,dx}{\pi\sqrt{r^2-x^2}(x^2+(\sqrt{r^2+1}+1+\sqrt{r^2-x^2})^2)},\ x\in(-r,r).
\end{align*}
To evaluate $I_\delta(\mu_{[-r,r]})$ explicitly, we use Theorem 5.1 of \cite[p. 124]{MR1485778} that gives
\begin{align*}
&\int \log|z-t|\,d\delta_i(t) + \int g_\Omega(z,t)\,d\delta_i(t) \\ &= \int \log|z-t|\,d\omega_{\Omega}(i,t) + \int g_\Omega(t,\infty)\,d\delta_i(t), \quad z\in\Omega.
\end{align*}
It follows that 
\begin{align*}
\int \log|z-t|\,d\omega_{\Omega}(i,t) = g_\Omega(z,i) + \log|z-i| - g_\Omega(i,\infty), \quad z\in\Omega.
\end{align*}
Since the left hand side is continuous at $z=i$, we obtain that
\begin{align*}
\int \log|i-t|\,d\omega_{\Omega}(i,t) = \lim_{z\to i} (g_\Omega(z,i) + \log|z-i|) - g_\Omega(i,\infty).
\end{align*}
Hence
\begin{align*}
I(\mu_{[-r,r]}) &= g_\Omega(i,\infty) + \int \log|i-t|\,d\omega_{\Omega}(i,t) = \lim_{z\to i} (g_\Omega(z,i) + \log|z-i|).
\end{align*}
Furthermore, it is well known that $g_\Omega(z,i)=\log|\Phi(z)|,\ z\in\Omega,$ see \cite[p. 113]{MR1334766}. This allows to compute the limit
\begin{align*}
\lim_{z\to i} (g_\Omega(z,i) + \log|z-i|) &= \lim_{z\to i} \log|\Phi(z)(z-i)| = \lim_{z\to i} \log|\Phi_2\left(\Phi_1(z)\right)(z-i)| \\ &= \log\left| \lim_{w\to w_0} (\overline{w}_0 w - 1)\, \lim_{z\to i} \frac{z-i}{w-w_0} \right|\quad (w=\Phi_1(z)) \\ &= \log \left|\frac{|w_0|^2-1}{\Phi_1'(i)}\right| = \log\frac{2\sqrt{r^2+1}}{r}.
\end{align*}
\end{proof}

\section{Potential-theoretic lower bounds on the Arakelov height}

\begin{proof}[Proof of Theorems~\ref{TotallyvadicThm} and \ref{TotallyvadicIntervalThm}]
We prove $(\ref{TotallyvadicBound1})$; the proofs of $(\ref{TotallyvadicBound2})$ and $(\ref{TotallyvadicBound3})$ are the same with trivial modifications.  Suppose we have a sequence $\{\alpha_k\}$ of distinct points in $\PP^1(L_S)$ with $\har(\alpha_k)\to\ell$ for some $\ell\in\RR$.  We must show that $\ell$ is greater than or equal to the right-hand side of $(\ref{TotallyvadicBound1})$.

For each point $\alpha_k$, let $[\alpha_k]_\infty$ be the Borel probability measure on $\PP^1(\CC)$ supported equally on the $\Gal(\Qbar/\QQ)$-conjugates of $\alpha_k$.  Passing to a subsequence, we may assume without loss of generality that the sequence of measures $\{[\alpha_k]_\infty\}$ converges weakly to a Borel probability measure $\nu_\infty$ on $\PP^1(\CC)$; this follows from Prokhorov's theorem  and the compactness of $\PP^1(\CC)$.

Similarly, for each place $p\in S$, let $[\alpha_k]_p$ be the Borel probability measure on $\PP^1(\QQ_p)$ supported equally on the $\Gal(\Qbar/\QQ)$-conjugates of $\alpha_k$.  Again passing to a finite number of subsequences, we may assume without loss of generality that for each $p\in S$, the sequence of measures $\{[\alpha_k]_p\}$ converges weakly to a Borel probability measure $\nu_p$ on $\PP^1(\QQ_p)$.

We have
\begin{equation}\label{LimInfLowerBound}
\begin{split}
\ell & = \lim_{k\to+\infty}\har(\alpha_k) \\
	& \geq \frac{1}{2}\liminf_{k\to+\infty}\sum_{v\in S\cup\{\infty\}}D_v(\alpha_k) \\
	& \geq \frac{1}{2}\sum_{v\in S\cup\{\infty\}}\liminf_{k\to+\infty} D_v(\alpha_k) \\
	& \geq \frac{1}{2}\sum_{v\in S\cup\{\infty\}}I_v(\nu_v) \\
	& \geq \frac{1}{4}+\frac{1}{2}\sum_{p\in S}\frac{p\log p}{p^2-1}
\end{split}
\end{equation}
as desired.  In $(\ref{LimInfLowerBound})$, the first inequality uses $(\ref{ArakelovDefEnergySum})$ and  the nonnegativity of the $D_p(\alpha)$ for $p\in M_{\QQ}\setminus(S\cup\{\infty\})$; the second inequality uses a basic property of limits; the third inequality uses \cite[Lemma 7.54]{MR2599526}, a general result about discrete approximations to energy integrals; and the fourth inequality uses $(\ref{FiliPetscheNonArch})$ and Theorem~\ref{thm:C}.
\end{proof}

\end{document}